\documentclass[11pt,twoside]{amsart}

\usepackage[latin1]{inputenc}
\usepackage[OT1]{fontenc}
\usepackage{amsmath}
\usepackage{amsthm}
\usepackage{amssymb}
\usepackage[all]{xy}
\usepackage{hyperref}

\DeclareMathOperator{\Spec}{Spec}
\DeclareMathOperator{\pr}{pr}
\DeclareMathOperator{\id}{id}

\DeclareMathOperator{\Hom}{Hom}

\DeclareMathOperator{\Coh}{Coh}

\DeclareMathOperator{\Ho}{H}

\DeclareMathOperator{\Hilb}{Hilb}

\DeclareMathOperator{\supp}{supp}

\DeclareMathOperator{\ord}{ord}

\DeclareMathOperator{\cone}{cone}
\DeclareMathOperator{\conj}{\mathsf{conj}}

\DeclareMathOperator{\Cent}{\mathsf C}

\DeclareMathOperator{\CY}{\mathsf{CY}}
\DeclareMathOperator{\HC}{\mathsf{HC}}
\DeclareMathOperator{\HH}{\mathsf{HH}}

\newcommand{\qqed}{\hspace*{\fill}$\Box$}

\newcommand{\Db}{{\rm D}^{\rm b}}

\newcommand{\Aut}{{\rm Aut}}

\DeclareMathOperator{\Pic}{{\rm Pic}}

\newcommand{\cI}{{\mathcal I}}

\newcommand{\IC}{\mathbb{C}}

\newcommand{\IP}{\mathbb{P}}

\DeclareMathOperator{\Inf}{\mathsf{Inf}}
\DeclareMathOperator{\Res}{\mathsf{Res}}

\DeclareMathOperator{\FM}{\mathsf{FM}}

\DeclareMathOperator{\MM}{\mathsf{M}}

\newcommand{\sym}{\mathfrak S}
\newcommand{\C}{\mathbb C}
\newcommand{\N}{\mathbb N}

\newcommand{\Z}{\mathbb Z}
\newcommand{\cA}{\mathcal A}

\newcommand{\cF}{\mathcal F}

\newcommand{\cE}{\mathcal E}

\newcommand{\cX}{\mathcal X}
\newcommand{\cY}{\mathcal Y}
\newcommand{\cZ}{\mathcal Z}

\newcommand{\cC}{\mathcal C}
\newcommand{\cD}{\mathcal D}
\newcommand{\cS}{\mathcal S}

\newcommand{\cP}{\mathcal P}
\newcommand{\cQ}{\mathcal Q}
\newcommand{\cT}{\mathcal T}
\newcommand{\cL}{\mathcal L}
\newcommand{\cK}{\mathcal K}

\newcommand{\alt}{\mathfrak a}

\newcommand{\reg}{\mathcal O}
\newcommand{\regcan}{\mathcal O^{\mathsf{can}}}

\newcommand{\eps}{\varepsilon}

\renewcommand{\P}{\mathbb P}
\renewcommand{\theta}{\vartheta}
\renewcommand{\rho}{\varrho}
\renewcommand{\phi}{\varphi}
\renewcommand{\_}{\underline{\,\,\,\,}}

\newtheorem{theorem}{Theorem}[section]
\newtheorem{prop}[theorem]{Proposition}
\newtheorem{lemma}[theorem]{Lemma}
\newtheorem{cor}[theorem]{Corollary}

\theoremstyle{definition}
\newtheorem{definition}[theorem]{Definition}

\newtheorem{remark}[theorem]{Remark}

\begin{document}

\title[Equivalences of equivariant...]{Equivalences of equivariant derived categories}

\author[A.\ Krug]{Andreas Krug}
\address{Andreas Krug, Mathematics Institute,
University of Warwick,
CV4 9EL,
Coventry, United Kingdom}
\email{a.krug@warwick.ac.uk}

\author[P.\ Sosna]{Pawel Sosna}
\address{Pawel Sosna, Fachbereich Mathematik der Universit\"at Hamburg,
Bundesstrasse 55,
20146 Hamburg, Germany}
\email{pawel.sosna@math.uni-hamburg.de}

\begin{abstract}

We investigate conditions for a Fourier-Mukai transform between derived categories of coherent sheaves on smooth projective stacks endowed with actions by finite groups to lift to the associated equivariant derived categories. As an application we give a condition under which a global quotient stack cannot be derived equivalent to a variety. We also apply our techniques to generalised Kummer stacks and symmetric quotients.
   
\end{abstract}
\maketitle

\section{Introduction}
If $Z$ is a smooth projective variety and $\Db(Z)$ its bounded derived category of coherent sheaves, it is interesting to understand the group of autoequivalences of $\Db(Z)$ or to investigate when $\Db(Z)$ is equivalent to $\Db(Z')$ for some other variety $Z'$. If the latter holds, it is reasonable to try and study what happens to the equivalence if we pass from $Z$ and $Z'$ to varieties which are closely connected to them geometrically. 

As an example, we can consider a smooth projective variety with torsion canonical bundle of order $n=\ord(\omega_X)$ and its canonical cover $p_X\colon \widetilde X\to X$, where $\widetilde X=\Spec(\reg_X\oplus\omega_X\oplus\ldots\oplus \omega_X^{n-1})$; see \cite{BM} or Section \ref{geomint} for details. There is a free action of $G=\mathbb{Z}/n\mathbb{Z}=\langle \tau_X\rangle$ on $\widetilde X$ such that $\widetilde X/G=X$. If $Y$ is a second smooth projective variety with torsion canonical bundle of the same order, Bridgeland and Maciocia, see \cite{BM}, show the following theorem:

\emph{Any equivalence $F\colon \Db(X)\cong \Db(Y)$ lifts to an equivalence $\Db(\widetilde X)\cong\Db(\widetilde Y)$. Conversely, if $\widetilde F\colon \Db(\widetilde X)\cong \Db(\widetilde Y)$ is an equivariant equivalence, that is, there exists an integer $k$ coprime to $n$ such that $\tau_Y^{k*}\circ \widetilde F\cong \widetilde F\circ \tau_X^*$, then $\widetilde F$ descends.}

In \cite{LombPopa}, there is a similar statement where the canonical bundle is replaced by a torsion bundle $L\in \Pic^0(X)$. 

In this paper we generalise the above results by using a strictly equivariant approach to lift and descent. In Section \ref{th3} we first forget about the geometric situation of coverings and give criteria for an equivalence between the derived categories of smooth projective stacks associated to normal projective varieties with only quotient singularities (compare \cite{Kawstack}) to lift to an equivalence between the categories linearised by certain finite subgroups of the groups of standard autoequivalences, see Theorem \ref{mainthm}. These results are achieved by a straightforward generalisation of the theory of equivariant Fourier-Mukai transforms as developed in \cite{Ploog-equiv} and they translate to the following statement (which combines Propositions \ref{descentthm} and \ref{liftthm} with Remark \ref{simplecyclic}) for descent and lift along quotients of smooth projective stacks by group actions generalising the criteria of \cite{BM} and \cite{LombPopa}.
See Definitions \ref{beta-equivariant}, \ref{mu-linearisable}, \ref{mu-equivariant} and Section \ref{geomint} for the notions used.

\begin{theorem}
Let $\pi\colon \widetilde\cX\to \cX$ and $\pi\colon \widetilde\cY\to \cY$ be Galois covers of smooth projective stacks with groups of deck transformations $H$ and $H'$, respectively. Let $\Phi=\FM_P\colon \Db(\widetilde\cX)\to \Db(\widetilde\cY)$ be a Fourier-Mukai transform. If $P\in \Db(\widetilde \cX\times \widetilde\cY)$ is $\mu$-linearisable for some isomorphism $\mu\colon H\cong H'$, then there is a $\hat \mu$-equivariant descent $\Psi\colon \Db(\cX)\to \Db(\cY)$. In addition, if $\Phi$ is an equivalence or fully faithful, so is $\Psi$. If $\Phi$ is spherical or $\P^n$ satisfying a technical assumption, then the same holds for $\Psi$. If $H$ and $H'$ are abelian, there is an analogous criterion for lift of Fourier-Mukai transforms along the Galois covers. If $H$ and $H'$ are cyclic, the condition that $\cP$ is $\mu$-linearisable can be replaced by the condition that $\Phi$ is $\mu$-equivariant.     
\end{theorem}

The paper is organised as follows. In Section 2 we collect background material on group actions, equivariant categories and Fourier-Mukai transforms. Section 3 is devoted to the study of lifts of FM transforms to equivariant categories. The achieved results are interpreted geometrically in Section 4. In Section \ref{Section-Applications} we give more specific applications. For instance, 
we describe conditions under which global quotient stacks cannot be derived equivalent to a smooth variety; see Proposition \ref{noFMpartners}. Furthermore, we study the derived category of symmetric quotients and closely related stacks.
In the appendix we investigate necessary conditions for lifts.

\smallskip
\textbf{Conventions.} Varieties or stacks are assumed to be smooth and projective unless stated otherwise and to be defined over the complex numbers. All functors between derived categories are assumed to be derived although this will not be reflected in the notation.

\smallskip
\textbf{Acknowledgements.} A.\ K.\ was financially supported by the research grant KR 4541/1-1 of the DFG and P.\ S.\ was partially financially supported by the RTG 1670 of the DFG. We thank David Ploog for his comments.

\section{Preliminaries}
\subsection{Group actions on categories}\label{groupactions}
Let $G$ be a finite group and $\cT$ a category. A \textit{categorical action} of $G$ on $\cT$ is the data of an autoequivalence 
$g^*\colon \cT\to \cT$ for every $g\in G$ together with isomorphisms $g^*h^*\cong (hg)^*$ satisfying a natural cocycle condition (see \cite{Del-cat} or \cite{Soscat} for details). Given such an action, a \textit{$G$-linearised} (or \textit{$G$-equivariant}) object is an object $E\in \cT$ together with a collection of isomorphisms $\lambda_g\colon E\to g^*E$ for $g\in G$ such that $\lambda_e=\id_E$ and such that the composition 
\[E\xrightarrow{\lambda_g}g^*E\xrightarrow{g^*\lambda_h}g^*h^*E\cong (hg)^*E\]
equals $\lambda_{hg}$ for every pair $g,h\in G$. Given two $G$-linearised objects $(E,\lambda)$ and $(F,\mu)$ we have a $G$-action on $\Hom_\cT(E,F)$ given by $\phi\cdot g:=\mu_g^{-1}\circ g^*(\phi)\circ \lambda_g$. We define the morphisms in the \textit{category of $G$-linearised objects $\cT^G$}, sometimes also denoted by $\cT_G$, to be the invariants under this action, i.e. 
\begin{align}\label{invhom}
\Hom_{\cT^G}((E,\lambda),(F,\mu)):=\Hom(E,F)^G
\end{align}
consists of morphisms $\phi\colon E\to F$ commuting with the linearisations. 

Let a finite group $G$ act on a $\IC$-linear category $\cT$ via $\IC$-linear autoequivalences. Then the group of characters
$\hat G=\Hom(G,\IC^*)$ acts on $\cT^G$ by tensor product. This means that a $G$-equivariant object $(E,\lambda)$ is sent to $(E,\lambda)\otimes \chi:= (E,\lambda')$ with $\lambda'_g=\chi(g)\cdot \lambda_g$. The action on morphisms is trivial.

   
%
For the proof of the following duality see \cite{Elagin-equiv}. Recall that a category is called \emph{Karoubian} if all idempotents have kernels.

\begin{prop}\label{abdual}
If $G$ is a finite abelian group acting on a Karoubian $\IC$-linear category $\cT$, then there is an equivalence $\cT\cong (\cT^G)^{\hat G}$.\qqed 
\end{prop}

We say that an object $E\in \cT$ is \textit{linearisable} if it admits a linearisation.
There is also the weaker notion of a \textit{G-invariant} object $E\in \cT$. This means that there are isomorphisms $E\cong g^*E$ without requiring any condition on their compositions. The relationship of the two notions for simple objects is as follows.

\begin{lemma}[\cite{Ploog-equiv}]\label{simple}
Let $G$ act on a $\IC$-linear category $\cT$ by $\IC$-linear automorphisms and let $E\in \cT$ be a simple object, i.e. $\Hom(E,E)=\IC$.
Then there exists a group cohomology class $[E]\in \Ho^2(G,\IC^*)$ with the property that $E$ is linearisable if and only if $[E]=0$. Furthermore, if $[E]=0$, then the set of isomorphism classes of $G$-linearisations of $E$ is a free $\hat G$-orbit under the $\hat G$-action described above.
\end{lemma}

\begin{proof}
This result is stated and proved in \cite[Lem.\ 1]{Ploog-equiv} in the special case that $\cT=\Db(X)$ for a smooth projective variety $X$ and $G\subset \Aut(X)$. The proof of the general result is exactly the same. 
\end{proof}

\begin{remark}\label{cysimple}
If $G=\Z/n\Z$ is cyclic, then $\Ho^2(G,\IC^*)=0$. Thus, in this case every $G$-invariant simple object is automatically linearisable.
\end{remark}
There exist natural functors between the category $\cT$ and the equivariant category $\cT^G$. Dropping the linearisations gives the \textit{restriction} (also called the \textit{forgetful}) functor $\Res\colon \cT^G\to \cT$, $(E,\lambda)\mapsto E$. 

Let $\cT=\cA$ be an additive (abelian) category and let the autoequivalences $g^*$ all be additive (exact). Then $\cA^G$ is again an additive (abelian) category (see \cite[Prop.\ 3.2]{Soscat}).
In this case there exists the \textit{inflation} functor $\Inf\colon \cA\to \cA^G$ defined as follows:
For $E\in \cA$ we have 
\[\Inf(E)=\bigoplus\limits_{g\in G}g^*E\]
equipped with the linearisation given by permuting the direct summands. If $G'\subset G$ is a subgroup, there is also the partial restriction functor $\Res_{G'}^G\colon \cT^G\to \cT^{G'}$ as well as a relative version of the inflation functor, namely 
\[\Inf^G_{G'}\colon \cA^{G'}\to \cA^G,\quad E\mapsto \bigoplus\limits_{[g]\in G/G'}g^*E\]
and the linearisation can be described explicitly; see \cite{Ploog-equiv}.

\begin{lemma}\label{ResInfadjunction}
The functor $\Inf_{G'}^G$ is left and right adjoint to $\Res^G_{G'}$.
\end{lemma}

\begin{proof}
See \cite[Lem.\ 3.7]{Elagin-equiv} for the case that $G'=1$. The general case is similar. 
\end{proof}

\begin{remark}
If $\cA$ is a symmetric monoidal category (see \cite[Ch.\ VII]{Maclane}), that is, there is a bifunctor $\otimes$ satisfying the usual properties, and if the group $G$ acts by monoidal equivalences, then $\cA^G$ inherits a monoidal structure. Indeed, if $(A,\lambda)$ and $(B,\mu)$ are linearised objects, then $\lambda\otimes\mu$ gives a linearisation of $A\otimes B$, since $\lambda_g\otimes\mu_g\colon A\otimes B\cong g^*(A)\otimes g^*(B)\cong g^*(A\otimes B)$. If, in addition, there is a unit with respect to $\otimes$ and $g^*$ sends the unit to itself for all $g\in G$, then $\cA^G$ also inherits a unit.
\end{remark}

\begin{remark}\label{equivtriangulated}
If $\cA$ is abelian with enough injective objects, a categorical action of $G$ on $\cA$ induces a categorical action of $G$ on $\Db(\cA)$. These actions are compatible in the sense that the categories $\Db(\cA^G)$ and $\Db(\cA)^G$ are equivalent; see \cite{Chen} or \cite{Elagin-equiv}. Note that if $G$ acts on a triangulated category $\cT$, there is a priori no reason for $\cT^G$ to be a triangulated category since cones are not functorial and there is no canonical way to define a linearisation of a cone of a morphism between equivariant objects in $\cT$; see \cite{Soscat} and \cite{Elagin-equiv} for an approach to circumvent this problem using differential graded enhancements.
\end{remark}

\subsection{Smooth projective stacks and Fourier-Mukai transforms}
The spaces we will work with are smooth projective stacks satisfying the assumptions of \cite{Kawstack}. This means that whenever we speak of a \textit{smooth projective stack} in the following, we mean a stack $\cX$ which is naturally associated with a normal projective variety with only quotient singularities; see \cite[Sect.\ 4]{Kawstack} for details. 

Let $\cX$ and $\cY$ be smooth projective stacks. For any object  $P\in \Db(\cX\times\cY)$ there is the associated \textit{Fourier-Mukai transform}
\[
 \FM_P:=\pr_{\cY*}\bigl(\pr_{\cX}^*(\_)\otimes P\bigr)\colon \Db(\cX)\to \Db(\cY).
\]
Let $\cZ$ be a third smooth projective stack and $Q\in \Db(\cY\times \cZ)$. Then $\FM_Q\circ \FM_P\cong \FM_{Q\star P}$, where 
\[
 Q\star P:=\pr_{\cX\cZ*}(\pr_{\cY\cZ}^*Q\otimes \pr_{\cX\cY}^*P)
\]
is the \textit{convolution product}.

Note that the kernel $P$ also induces a second Fourier-Mukai transform $\Db(\cY)\to \Db(\cX)$ in the opposite direction. To avoid the risk of confusing the two FM transforms we will sometimes write them as $\FM_P^{\cX\to \cY}\colon \Db(\cX)\to \Db(\cY)$ and $\FM_P^{\cX\gets \cY}\colon \Db(\cY)\to \Db(\cX)$, respectively. 

\begin{lemma}\label{Orlovlem}
Let $\cX_1$, $\cX_2$, $\cY_1$, and $\cY_2$ be smooth projective stacks, and let $R\in\Db(\cX_1\times \cX_2)$ and $P_i\in\Db(\cX_i\times \cY_i)$ for $i=1,2$ be arbitrary. We consider the exterior tensor product $P_1\boxtimes P_2\in \Db((\cX_1\times \cX_2)\times (\cY_1\times \cY_2))$ and set $S:=\FM_{P_1\boxtimes P_2}(R)\in \Db(\cY_1\times \cY_2)$. Then
\[
 \FM_S^{\cY_1\to \cY_2}\cong \FM_{P_2}^{\cX_2\to \cY_2}\circ \FM_R^{\cX_1\to \cX_2}\circ \FM_{P_1}^{\cX_1\gets \cY_1}\,.
\]
\end{lemma}
\begin{proof}
For varieties this is \cite[Prop.\ 2.1.6]{Orlsurvey}. The proof in the case of stacks is exactly the same. 
\end{proof}

\begin{remark}

Any stack in our sense is actually a global quotient stack of the form $[X/G]$, where $X$ is a quasi-projective scheme and $G$ a reductive algebraic group acting linearly; see, for example, \cite[Prop.\ 5.1]{Kresch}. Note that sheaves on such a quotient stack are the $G$-equivariant sheaves on $X$. However, since $G$ is not necessarily finite, the condition to be equivariant has to be reformulated, see, for instance, \cite[Subsect.\ 4.2]{HuyLehn}.

In all of our applications, the stacks will even be global finite quotient stacks, that is, smooth projective stacks of the form $\cX=[X/G]$ for $X$ a smooth projective variety and $G\subset \Aut(X)$ a finite subgroup. In this case there are equivalences $\Coh(\cX)\cong \Coh_G(X)$ and, accordingly, $\Db(\cX)\cong \Db_G(X):=\Db(\Coh_G(X))$. Under these equivalences, the Fourier-Mukai transforms are given by the equivariant Fourier-Mukai transforms of \cite{Ploog-equiv}.  
\end{remark}

\section{Lifts of equivalences to equivariant categories}\label{th3}
\subsection{Linearisations by standard autoequivalences}
Let $\cX$ be a smooth projective stack. Every automorphism $\phi\in \Aut(\cX)$ (note that this is really a group in our setting and not a possibly more complicated categorical object as for general stacks) gives the pull-back autoequivalence $\phi^*\colon\Coh(\cX)\to \Coh(\cX)$. Furthermore, every line bundle $\cL\in\Pic(\cX)$ induces the autoequivalence $\MM_\cL:=(\_)\otimes \cL\colon \Coh(\cX)\to \Coh(\cX)$. 

We set $A(\cX):=\Pic(\cX)\rtimes \Aut(\cX)$. 
The product structure in $A(\cX)$ is given by $(\cL,\phi)\cdot(\cK,\psi)=(\psi^*\cL\otimes\cK,\phi\circ \psi)$.

\begin{definition}
There is an automorphism $c$ of $A(\cX)$ defined as follows. If $h=(\cL,\phi)\in A(\cX)$, then $c(h):=\bar h:=(\cL^{-1},\phi)$. We say that a subgroup $H\subset A(\cX)$ is \textit{$c$-invariant} if $\bar H=H$. 
Furthermore, for $h=(\cL, \phi)$, we set $p(h):=(\reg,\phi)$.  
\end{definition}

\begin{remark}\label{reverse}
For $\cL\in \Pic(\cX)$ the FM kernel of $\MM_\cL$ is given by the pushforward along the diagonal $\Delta_*\cL\in \Db(\cX\times \cX)$ and for $\phi\in \Aut(\cX)$ the FM kernel of the pushforward $\phi_*$ is the structure sheaf of the graph $\reg_{\Gamma_\phi}\in \Db(\cX\times \cX)$. Note that the FM tranform in the opposite direction $\FM^{\cX\gets \cX}_{\Delta_*\cL}$ is again $\MM_\cL$, while $\FM_{\reg_{\Gamma_\phi}}^{\cX\gets \cX}$ is given by $\phi^*$, the inverse of $\phi_*$. This different behaviour can be seen as the reason for the occurrence of the automorphism $c$.
\end{remark}
 
In the following we will consider finite $c$-invariant subgroups $H\subset A(\cX)$ and assume that they \textit{act categorically} on $\Coh(\cX)$. With this we mean that there is a categorical action of $H$ on $\Coh(\cX)$ such that $h^*=\MM_{\cL}\circ \phi^*$ for $h=(\cL,\phi)\in H$. Clearly, any $H\subset\Aut(\cX)$ acts categorically, since the isomorphisms $\phi^*\circ \psi^*=(\psi\circ \phi)^*$ satisfy cocycle conditions. One can also check that every $H\subset\Pic(\cX)$ acts categorically; compare \cite{Elagin-equiv} and note that the proof given there works for line bundles on smooth projective stacks as well.

In the applications we will always have either $H\subset \Aut(\cX)$ or $H\subset \Pic(\cX)$. The main reason that we nevertheless choose to work with general $c$-invariant subgroups acting categorically is to avoid dealing with two cases. 

Given a $c$-invariant subgroup $H\subset A(\cX)$ acting categorically, we denote by $\Coh_H(\cX):=\Coh(\cX)^H$ the $H$-equivariant category. Its objects are sheaves  $E\in \Coh(\cX)$ together with a linearisation $\lambda$ consisting of isomorphisms $\lambda_h\colon E\xrightarrow \cong h^*E=\cL\otimes \phi^*E$ for $h=(\cL,\phi)\in H$.  
\subsection{Fourier-Mukai transforms of $\mu$-type} 
Let $\cX$ and $\cY$ be smooth projective stacks and $H\subset A(\cX)$ and $H'\subset A(\cY)$ be finite subgroups.
There is an induced $H\times H'$-action on $\Db(\cX\times \cY)$ given for
$h\in H$ and $h'\in H'$ by $(h\times h')^*:=\MM_{\cL\boxtimes \cL'}\circ (\phi\times \phi')^*$.

\begin{definition}\label{mu-linearised}
A group isomorphism $\mu\colon H\to H'$ is called a \emph{$c$-isomorphism} if $c\circ\mu=\mu\circ c$. Note that if $H\subset \Aut(\cX)$ and $H'\subset \Aut(\cY)$ or $H\subset \Pic(\cX)$ and $H'\subset \Pic(\cY)$ every isomorphism $\mu\colon H\to H'$ is a $c$-isomorphism. 

Given such a $c$-isomorphism, we let $h\in H$ act on $\Coh(\cX\times \cY)$ by $(\bar h\times \mu(h))^*$. We denote the category of sheaves linearised with respect to this $H$-action by $\Coh_\mu(\cX\times \cY)$ and its derived category by $\Db_\mu(\cX\times \cY)$. The objects in these categories are called \textit{$\mu$-linearised}. 
\end{definition}

For $\cP=(P,\nu)\in \Db_\mu(\cX\times \cY)$ there is an associated \textit{Fourier--Mukai transform of $\mu$-type} given by 
\[\FM^\mu_{\cP}:=\pr_{\cY*}\left(\pr_\cX^*(\_)\otimes \cP   \right)\colon \Db_H(\cX)\to \Db_{H'}(\cY).\]  
For $\cE=(E,\lambda)\in \Db_{H}(\cX)$, and $h'\in H'$ there are the isomorphisms 
\[
\alpha_{h'}:=\pr_\cX^*\lambda_{\mu^{-1}(h')}\otimes \nu_{\mu^{-1}(h')}\colon \pr_\cX^*E\otimes P\xrightarrow\cong (p\circ\mu^{-1}(h')\times h')^*(\pr_\cX^*E\otimes P).
\]
Pushing forward by $\pr_{\cY*}$ gives the $H'$-linearisation of $\pr_{\cY*}(\cP\otimes \pr_{\cX}^*\cE)=\FM^\mu_{\cP}(\cE)$. 

The sheaf $\reg_{\Delta\cX}\in \Db(\cX\times \cX)$ has a canonical $\id_H$-linearisation so that the identity $\id\colon \Db_H(\cX)\to \Db_H(\cX)$ is a Fourier-Mukai transform of $\id_H$-type. We denote $\reg_{\Delta\cX}$ equipped with this linearisation by $\regcan_{\Delta\cX}$. Here, one can see the need to let $h$ act by
$(\bar h\times h)^*$ and not by $(h\times h)^*$. More generally, for a character $\rho\in \hat H$, the tensor product $\MM_\rho\colon \Db_H(\cX)\to \Db_H(\cX)$ is the FM transform of $\id_H$-type with kernel $\regcan_{\Delta\cX}\otimes \rho$.

Let $\mu'\colon H'\rightarrow H''$ be another $c$-isomorphism with $H''\subset A(\cZ)$ for a third smooth projective stack $\cZ$. Then for $\cP'\in \Db_{\mu'}(\cY\times\cZ)$ the composition $\FM^{\mu'}_{\cP'}\circ \FM^\mu_\cP$ is a Fourier-Mukai transform of $\mu'\circ \mu$-type with kernel given by the convolution product
\begin{align}\label{muconv}\cP'\star\cP=\pr_{\cX\cZ*}(\pr_{\cX\cY}^*\cP\otimes \pr_{\cY\cZ}^*\cP')\in \Db_{\mu'\circ \mu}(\cX\times \cZ)\,.\end{align}

\begin{definition}\label{equiliftdef}
Let $F\colon \Db(\cX)\to \Db(\cY)$ and $\widetilde F\colon \Db_H(\cX)\to \Db_{H'}(\cY)$ be exact functors. Then $\widetilde F$ is called a \textit{lift} of $F$ if the following two conditions hold
\begin{align*}
F\circ \Res&\cong \Res\circ \widetilde F\colon \Db_H(\cX)\to \Db(\cY),\\
\Inf\circ F&\cong \widetilde F\circ \Inf\colon \Db(\cX)\to\Db_{H'}(\cY). 
\end{align*}
\end{definition}

\begin{definition}\label{beta-equivariant}
Let $\beta\colon \hat H'\to\hat H$ be an isomorphism of groups. A lift $\widetilde F\colon \Db_H(\cX)\to \Db_{H'}(\cY)$ of  $F\colon \Db(\cX)\to \Db(\cY)$ is called \textit{$\beta$-equivariant} if $\widetilde F\circ \MM_{\beta(\rho')}\cong \MM_{\rho'}\circ \widetilde F$ for all $\rho'\in\hat H'$. Here, $\MM_{\rho'}$ denotes the tensor product by the character in the equivariant category. 
\end{definition}     
By the very construction of the FM transforms of $\mu$-type one gets the following
\begin{lemma}\label{liftlem}
Let $\cP=(P,\nu)\in \Db_\mu(\cX\times \cY)$ for some isomorphism $\mu\colon H\to H'$. Then $\widetilde F=\FM^\mu_\cP\colon \Db_H(\cX)\to \Db_{H'}(\cY)$ is a $\hat \mu$-equivariant lift of $F=\FM_P\colon\Db(\cX)\to \Db(\cY)$.\qqed 
\end{lemma}
Before we state our next result, which is a generalisation of \cite[Lem.\ 5]{Ploog-equiv}, recall that every $P\in \Db(\cX\times \cY)$ has a right and a left adjoint kernel given by $P^L:=P^\vee\otimes \pr_{\cY}^*\omega_\cY[\dim\cY]$ and $P^R:=P^\vee\otimes \pr_{\cX}^*\omega_\cX[\dim\cX]$; see \cite{CaWi}.  

\begin{prop}\label{equilift}
Let $\cP=(P,\nu)\in \Db_\mu(\cX\times \cY)$ for some isomorphism $\mu\colon H\to H'$, $F=\FM_P$, and $\widetilde F=\FM^\mu_\cP$. Then 
\begin{enumerate}
 \item $F$ is fully faithful $\Longrightarrow$ $\widetilde F$ is fully faithful.
\item $F$ is an equivalence $\Longrightarrow$ $\widetilde F$ is an equivalence.
\end{enumerate}
\end{prop}
\begin{proof}
The right adjoint kernel $P^R=P^\vee\otimes \pr_\cX^*\omega_\cX[\dim \cX]\in \Db(\cY\times \cX)$ has an induced $\mu^{-1}$-linearisation $\widetilde \nu$ given by 
$
\widetilde \nu_{\mu(h)}=(\nu_h^\vee)^{-1}\otimes \pr_{\cX}^*\lambda_{p(h)}[\dim\cX]
$
where $\lambda$ is the natural $p(H)$-linearisation of the canonical bundle given by the pushforward of $\dim\cX$-forms. We denote $P^R$ equipped with this linearisation by $\cQ\in \Db_{\mu^{-1}}(\cY\times \cX)$. The convolution product (\ref{muconv}) is compatible with the restriction functor. So if $F$ is fully faithful, we have $\Res(\cQ\star\cP)=P^R\star P=\reg_{\Delta\cX}$. 
We have $\Hom(\reg_{\Delta\cX}, \reg_{\Delta\cX})\cong \Gamma(\cX,\reg_\cX)=\C$ which means that $\reg_{\Delta\cX}$ is simple.
Thus, $\cQ\star\cP=\regcan_{\Delta\cX}\otimes \rho$ for some $\rho\in \hat H$ by Lemma \ref{simple}. Hence, 
$\FM^{\mu^{-1}}_{\cP^R}\circ \widetilde F=\MM_\rho$ is an equivalence which proves that $\widetilde F$ is fully faithful. If $F$ is an equivalence, we get similarly that $\widetilde F\circ \FM^{\mu^{-1}}_{\cP^R}$ is an equivalence, too. It follows that $\widetilde F$ is an equivalence. 
\end{proof}

\subsection{Monoidal and $X$-linear autoequivalences}

The reason that the calculus of Fourier-Mukai transforms of $\mu$-type works for subgroups of $A(\cX)$ is that the pushforwards along automorphisms are monoidal, i.e.\ $F(A\otimes B)\cong F(A)\otimes F(B)$ and $F(\reg_X)=\reg_X$ for $F=\phi_*$, and tensor products by line bundles are $\cX$-linear, i.e. $G(A\otimes B)=G(A)\otimes B$ for $G=\MM_\cL$. 

If we consider varieties for simplicity, these are actually the only monoidal and linear autoequivalences:

\begin{prop}
 Let $X$ and $Y$ be smooth projective varieties.
 \begin{enumerate}
  \item Every $X$-linear autoequivalence of $\Db(X)$ is of the form $\MM_L[m]$ for some $L\in \Pic X$ and $m\in \Z$.
  \item Every monoidal equivalence $\Db(X)\cong \Db(Y)$ is of the form $\phi_*$ for some isomorphism $\phi\colon X\to Y$.
 \end{enumerate}
\end{prop}
\begin{proof}
Let $G\in \Aut(\Db(X))$ be $X$-linear, $x\in X$, and $E:=G(k(x))$. For $x\neq x'\in X$ we have
$E\otimes k(x')\cong G(k(x)\otimes k(x'))\cong 0$. Thus, $\supp E=\{x\}$; see \cite[Ex.\ 3.30]{Huy-book}. Since $E$ is point-like, we have $E=k(x)[m]$ for some $m\in \Z$; see \cite[Lem.\ 4.5]{Huy-book}. It follows that $G=\MM_L[m]$ for some line bundle $L\in \Pic X$; see \cite[Cor.\ 5.23 \& 6.14]{Huy-book}. 

Let $F\colon \Db(X)\to \Db(Y)$ be a monoidal equivalence, $x\in X$, and $E:=F(k(x))$. For every $A\in \Db(X)$ we have $k(x)\otimes A\cong k(x)\otimes_\C V$ for some graded vector space $V\in \Db(\C)$ (namely $V=i_x^*E$ where $i_x$ is the embedding of the point). It follows for every $B\in \Db(Y)$ that
\[
 E\otimes B\cong F(k(x)\otimes F^{-1}(B))\cong F(k(x)\otimes_\C V)\cong E\otimes_\C V 
\]
for some $V\in \Db(\C)$. 
In particular, for every $B\in \Db(Y)$ we have $\supp(E\otimes B)=\supp E$ or $\supp(E\otimes B)=0$.
Since $\supp(E\otimes k(y))\subset \{y\}$ for $y\in Y$, the support of $E$ consist of a single point. Again, since $E$ is point-wise, $E=k(y)[m]$ for some $y\in Y$ and $m\in \Z$. It follows again by \cite[Cor.\ 5.23 \& 6.14]{Huy-book} that $F=\MM_L\circ \phi_*[m]$ for some $L\in \Pic X$ and some isomorphism $\phi\colon X\to Y$. Since $F$ is monoidal, $F(\reg_X)\cong \reg_Y$. Thus, $L=\reg_Y$ and $m=0$.
\end{proof}
\begin{remark}
 Part (ii) of the above Proposition gives a quick proof that the derived category of a smooth projective variety together with its monoidal structure determines the variety (up to isomorphism); see \cite{BalSpec} for a constructive proof and a more general statement.
\end{remark}

\subsection{Lifts of adjunctions}
\begin{definition}
Let $\cP\in \Db_{\mu}(\cX\times \cY)$ and $\cQ\in \Db_{\mu^{-1}}(\cY\times \cX)$. Then $\cQ$ is said to be a \textit{right adjoint kernel} of $\cP$ (or, equivalently, $\cP$ is a \textit{left adjoint kernel of $\cQ$}), in short $\cP\dashv\cQ$, if there are morphisms $\eta\colon \regcan_{\Delta\cX}\to \cQ\star \cP$ and $\eps\colon \cP\star \cQ\to \regcan_{\Delta\cY}$ with $(\eps\star \id_{\cP})\circ(\id_\cP\star \eta)=\id_\cP$ and $(\id_{\cQ}\star \eps)\circ(\eta\star \id_{\cQ})=\id_{\cQ}$. 
The morphisms $\eta$ and $\eps$ are called the \textit{unit} and \textit{counit} of the adjunction. 
\end{definition}

Clearly, an adjunction of kernels induces an adjunction between the associated Fourier-Mukai transforms $\FM^{\mu}_\cP\dashv \FM^{\mu^{-1}}_\cQ$.    

Let $\cZ$ be a third smooth projective stack, $H''\subset A(\cZ)$, and $\mu'\colon H''\cong H$ a $c$-isomorphism. Let $\cP\dashv\cQ$ be as in the definition. For $\cE\in \Db_{\mu'}(\cZ\times \cX)$ and $\cF\in \Db_{\mu\circ \mu'}(\cZ\times \cY)$ the unit and counit of the adjunction $\cP\dashv \cQ$ induce an isomorphism 
\begin{align}\label{adjointformula}
\Hom_{\Db_{\mu\circ \mu'}(\cZ\times \cY)}(\cP\star \cE,\cF)\cong \Hom_{\Db_{\mu'}(\cZ\times \cX)}(\cE,\cQ\star \cF)\,.
\end{align}
In particular, there is the formula
\begin{align}\label{RFformula}
\Hom_{\Db_{\mu}(\cX\times \cY)}(\cP,\cP)\cong \Hom_{\Db_{\id_H}(\cX\times \cX)}(\regcan_{\Delta\cX},\cQ\star \cP)\,.
\end{align}

\begin{lemma}\label{liftadj}
Let $P\in \Db(\cX\times\cY)$ be simple.
Then for every $\mu$-linearisation $\nu$ of $P$, there exist $\mu^{-1}$-linearisations $\nu^L$ and $\nu^R$ of $P^L$ and $P^R$ such that $\cP^L:=(P^L,\nu^L)\dashv \cP:=(P,\nu)\dashv\cP^R:=(P^R,\nu^R)$. 
\end{lemma}
\begin{proof}
Let $\eta\colon \reg_{\Delta\cX}\to P^R\star P$ and $\eps\colon P\star P^R\to \reg_{\Delta\cY}$ be the unit and counit of the adjunction $P\dashv P^R$ and let $\nu$ be a linearisation of $P$.
Note that $P$ being simple is equivalent to $\Hom(\reg_{\Delta\cX},P^R\star P)=\mathbb{C}$ by  (\ref{RFformula}). 
Furthermore, $P^R$ is always $\mu^{-1}$-linearisable; see the proof of Proposition \ref{equilift}.
Let $\nu'$ be any $\mu^{-1}$-linearisation of $P^R$. 

Consider the $H$-action on $\Hom(\reg_{\Delta\cX},P^R\star P)=\mathbb{C}$ induced by the canonical linearisation of $\reg_{\Delta\cX}$ and the linearisation $\nu'\star \nu$ of $P^R\star P$. 
The $H$-action is given by some character 
$\rho\in \hat H$. Thus, after replacing $\nu'$ by $\nu^R:=\nu'\otimes \rho$, the $H$-action on $\Hom(\reg_{\Delta\cX},P^R\star P)$ becomes trivial which allows us to set $\widetilde \eta:=\eta\colon \regcan_{\Delta\cX}\to \cP^R\star \cP$.
We also set $\widetilde \eps=\frac 1{|H|}\sum_{h'\in H'}h'\cdot \eps$ where $h'\cdot \eps$ is the image of $\eps$ under the $H'$-action on $\Hom(P\star P^R,\reg_{\Delta\cY})$ induced by the linearisation $\nu\star \nu^R$ and the canonical linearisation of $\reg_{\Delta\cY}$. The equivariance of $\eta$ and $\id_P$ yield
\begin{align*}(\widetilde\eps\star \id_{\cP})\circ(\id_\cP\star \widetilde\eta)&=\frac1{|H|}\sum_{h\in H}(\mu(h)\cdot \eps\star \id_{\cP})\circ(\id_\cP\star \eta)\\&=
\frac1{|H|}\sum_{h'\in H'}h'\cdot[(\eps\star \id_{\cP})\circ(\id_\cP\star \eta)]\\&=
\frac1{|H|}\sum_{h'\in H'}h'\cdot \id_P\\&=\id_\cP\,.\end{align*}
Similarly, we get $(\id_{\cP^R}\star \widetilde\eps)\circ(\widetilde\eta\star \id_{\cP^R})=\id_{\cP^R}$. The proof of the existence of an adjunction $\cP^L\dashv \cP$ is analogous.
\end{proof}
\subsection{Lifts of spherical and $\P^n$-functors}
\begin{definition}
An object $\cP\in \Db_\mu(\cX\times \cY)$ with left and right adjoints $\cP^L$ and $\cP^R$ is said to be a \textit{spherical kernel} if the cone $\cC=\cone(\eta)$ of the unit is the kernel of an equivalence, called the \textit{cotwist kernel} of $F=\FM^\mu_\cP$, and the composition
\begin{align}
 \phi\colon \cP^R\xrightarrow{\cP^R\star\eta_L}\cP^R\star \cP\star \cP^L\xrightarrow{\beta\star  \cP^L}\cC\star \cP^L\,,
\end{align}
(where $\eta_L$ is the counit of the adjunction $\cP^L\dashv \cP$, $\beta$ is from the triangle $\regcan_{\Delta\cX}\xrightarrow{\eta} \cP^R\star \cP\xrightarrow\beta \cC$ and we abuse notation by writing $\cP^R$ for $\id_{\cP^R}$ and similarly for $\cP^L$) is an isomorphism.
\end{definition}

If $\cP$ is spherical, the associated \textit{twist} $\cT=\cone(\eps)\in \Db_{\id_{H'}}(\cY\times \cY)$ given by the counit of the adjuction $\cP\dashv\cP^R$ is the kernel of an autoequivalence; see \cite{Rou}, \cite{Addington}, \cite{Anno-Logvinenko}.

Furthermore, following \cite{Addington} we have

\begin{definition} 
An object $\cP\in\Db_\mu(\cX\times \cY)$ is a \textit{$\P^n$-kernel} if the following three conditions hold:
\begin{enumerate}
 \item There is an isomorphism $\alpha\colon \cP^R\star \cP\cong \regcan_{\Delta\cX}\oplus \cD\oplus \cD^{\star 2}\oplus\ldots\oplus \cD^{\star n}$ for some autoequivalence kernel $\cD\in \Db_{\id_H}(\cX\times \cX)$ which is then called the \textit{$\P$-cotwist kernel}. The components of the isomorphism $\alpha$ are denoted by $\alpha_i\colon \cP^R\star \cP\to \cD^{\star i}$. 
\item The composition
\begin{align}
 \psi\colon \cP^R\xrightarrow{\cP^R\star\eta_L}\cP^R\star \cP\star \cP^L\xrightarrow{\alpha_n\star  \cP^L}\cD^{\star n}\star \cP^L\,,
\end{align}
is an isomorphism.
\item The compositions
\[c_{ij}\colon \cD\star \cD^{\star i-1}\xrightarrow{\alpha^{-1}_1\star \alpha^{-1}_{i-1}} \cP^R\star\cP\star \cP^R\star \cP\xrightarrow{\cP^R\star \eps\star \cP}\cP^R\star\cP\xrightarrow{\alpha_j}\cD^{\star j}\]
are isomorphisms for $i=j$ and zero for $i<j$ (and arbitrary for $i>j$).
\end{enumerate}
\end{definition}

For $\P^n$-kernels there also is an associated twist which is an autoequivalence of $\Db_{H'}(\cY)$; see \cite[Sect.\ 3.3]{Addington}. 
Clearly, the above definitions include the case that $H=H'=1$.
The Fourier-Mukai transforms associated to spherical and $\P^n$-kernels are called \textit{spherical} and \textit{$\P^n$-functors}, respectively.

In Proposition \ref{Plift} below we will assume that our spherical and $\P^n$-functors satisfy the conditions
\begin{align}\label{sphericalcond}
 &\Hom(\regcan_{\Delta\cX},\cC)=0\quad&\text{for $\cP$ a spherical kernel,}\\\label{Pcond}
 &\Hom(\regcan_{\Delta\cX},\cD^{\star i})=0\quad \forall\, 1\le i\le n\quad&\text{for $\cP$ a $\P^n$-kernel.}
\end{align}
These conditions are satisfied by all spherical and $\P^n$-kernels the authors are aware of. 
Note that every spherical kernel satisfying (\ref{sphericalcond}) as well as every $\P^n$-kernel satisfying (\ref{Pcond}) is simple by (\ref{RFformula}). Furthermore, under condition (\ref{Pcond}) it is automatic that $c_{ij}=0$ for $i<j$.

\begin{remark}
The above axiom (ii) appears to be slightly stronger than in \cite{Addington} where it is only required that there is any isomorphism $\cP^R\cong \cD^{\star n}\star \cP^L$ (and in the spherical case it is only required that there is any isomorphism $\cP^R\cong \cC\star \cP^L$).
However, if $\cP$ is simple and satisfies the condition
\begin{align}\label{sphericalcond2}
 &\Hom(\regcan_{\Delta\cX},\cC)=0=\Hom(\cC^{\star-1},\cC)\quad&\text{for $\cP$ a spherical kernel}\\\label{Pcond2}
 &\Hom(\cD^{\star i},\cD^{\star n})=0\quad \forall\, -n\le i\le n-1\quad&\text{for $\cP$ a $\P^n$-kernel}
\end{align}
(where $\cC^{\star-1}$ denotes the kernel of the autoequivalence which is inverse to $\cC$) as most known spherical and $\P^n$-functors do, both definitions are equivalent. Note that there are spherical functors satisfying \ref{Pcond} but \ref{Pcond2}. 
Indeed, let $\cP\in \Db_\mu(\cX\times \cY)$ satisfy axiom (i) of a $\P^n$-functor and let there be an isomorphism $\theta\colon\cP^R\cong \cD^{\star n}\star \cP^L$. Then, for $\psi\colon \cP^R\to \cD^{\star n}\star \cP^L$ to be an isomorphism it is sufficient that $\psi$ is non-zero because $\cP^R$ is simple. For $i<j$ set $\cD^{\star[i,j]}:=\cD^{\star i}\oplus\cD^{\star i+1}\oplus\ldots\oplus \cD^{\star j}$ and consider the triangle 
\[\cD^{\star[0, n-1]}\star \cP^L\to \cP^R\star \cP\star \cP^L\xrightarrow{\alpha_n\star \cP^L} \cD^{\star n}\star\cP^L\,.\]
Under the assumption $\psi=(\alpha_n\star \cP^L)\circ(\cP^R\star \eta_L)=0$ it would follow that $\cP^R\star \eta_L$ factors over $\cD^{\star[0, n-1]}\star \cP^L$. Thus, the identity 
$\id_{\cP^R\star \cP}= (\cP^R\star\cP\star \eps_L)\circ (\cP^R\star \eta_L\star\cP)$ would factor over 
\begin{align*}
 \cD^{\star[0,n-1]}\star\cP^L\star \cP\cong \cD^{\star[0,n-1]}\star\cD^{\star-n}\star \cP^R\star \cP\cong \cD^{\star[0,n-1]}\star\cD^{\star-n}\star \cD^{\star[0,n]}
\end{align*}
which is impossible by condition (\ref{Pcond2}).
\end{remark}

\begin{prop}\label{Plift}
Let $P\in \Db(\cX\times \cY)$ be a spherical kernel satisfying (\ref{sphericalcond}) (a $\P^n$-kernel satisfying (\ref{Pcond})) which allows a $\mu$-linearisation $\nu$. Then $\cP:=(P,\nu)$ is again a spherical ($\P^n$-) kernel and also satisfies the respective condition. 
\end{prop}

\begin{proof}
 We give the proof only in the case that $P$ is a $\P^n$-kernel since the proof in the spherical case is similar. Let $P$ be a $\P^n$-kernel with $\P$-cotwist kernel $D$ and an isomorphism $\alpha\colon P^R\star P\cong \reg_{\Delta\cX}\oplus D\oplus\ldots\oplus D^{\star n}$ and let $\nu$ be a $\mu$-linearisation of $P$. For $i<j$ and $h\in H$ there do not exist non-zero morphisms $D^{\star i}\to (\bar h\times \mu(h))^*D^{\star j}$. Indeed, applying $(\bar h^{-1}\times \mu(h)^{-1})^*\circ \star D^{\star -i}$ to such a morphism would yield a non-zero morphism 
 $\reg_{\Delta\cX}\cong (\bar h^{-1}\times \mu(h^{-1}))^*\reg_{\Delta\cX}\to D^{\star j-i}$ contradicting (\ref{Pcond}).
 It follows that the linearisation $\nu^R\star \nu$ of $P^R\star P$ (see Lemma \ref{liftadj}) induces via the isomorphism $\alpha$ linearisations of the $D^{\star i}$. We set $\cP:=(P,\nu)$, $\cP^R:=(\cP^R,\nu^R)$, and denote for $i=1,\ldots,n$ the $D^{\star i}$ equipped with the induced linearisations by $\cD_i$, so that $\alpha$ gives an isomorphism $\cP^R\star \cP\cong\regcan_{\Delta\cX}\oplus\cD_1\oplus\ldots\oplus\cD_n$ in $\Db_{\id_H}(\cX\times \cX)$.
 The $\cD_i$ are equivalence kernels by Proposition \ref{equilift}(ii).
 By the definition of the linearisations of the $D^{\star i}$ the projections $\alpha_i \colon P^R\star P\to D^{\star i}$ are equivariant. Since the $c_{ii}$ are given by a composition of the $\alpha_i$, their inverses, and counits of adjunctions, the $c_{ii}$ are equivariant, too. Hence, they can be seen as isomorphisms $c_{ii}\colon \cD_1\star \cD_{i-1}\cong \cD_{i}$ in $\Db_{\id_H}(\cX\times \cX)$. This shows by induction that $\cD_i\cong \cD^{\star i}$, where $\cD:=\cD_1$, so that $\cP$ satisfies condition (i) of a $\P^n$-kernel. It also shows that condition (iii) of a $\P^n$-kernel holds. The morphism $\psi$ associated to $\cP$ is just the same as the morphism $\psi$ associated to $P$, since also the isomorphism $\alpha$ is the same in the equivariant and the non-equivariant case. In particular, $\psi$ is an isomorphism and $\cP$ fulfils condition (ii) of a $\P$-functor. Finally, $\Hom(\regcan_{\Delta\cX},\cD^{\star i})\subset \Hom(\reg_{\Delta\cX},D^{\star i})$ so that condition (\ref{Pcond}) is still satisfied.     
\end{proof}

Summarising the results of this section, we have

\begin{theorem}\label{mainthm}
Let $\cX$ and $\cY$ be smooth projective stacks, let $H\subset A(\cX)$ and $H'\subset A(\cY)$ be finite subgroups acting categorically and let $P\in \Db(\cX\times \cY)$. If $P$ admits a $\mu$-linearisation $\nu$ for some isomorphism $\mu\colon H\to H'$, then $F=\FM_P\colon \Db(\cX)\to \Db(\cY)$ lifts to $\widetilde{F}=\FM_{(P,\nu)}^\mu\colon \Db_H(\cX)\to \Db_{H'}(\cY)$. If $F$ is fully faithful (an equivalence, a spherical kernel satisfying (\ref{sphericalcond}) or a $\P^n$-kernel satisfying (\ref{Pcond})), then the same holds for $\widetilde{F}$. 
\end{theorem}

\begin{proof}
This is just a combination of Lemma \ref{liftlem} together with Propositions \ref{equilift} and \ref{Plift}.
\end{proof}

\subsection{Equivariance of functors vs.\ equivariance of kernels}
\begin{definition}\label{mu-linearisable}
Let $\mu\colon H\to H'$ be a $c$-isomorphism and $P\in \Db(\cX\times \cY)$. We say that $P$ is \textit{$\mu$-linearisable} if it admits a $\mu$-linearisation, i.e. if there is an object $\cP=(P,\nu)\in \Db_\mu(\cX\times \cY)$ such that $\Res\cP=P$.
There is also the weaker notion of \textit{$\mu$-invariance} of $P$ which means that there are for $h\in H$ isomorphisms
$P\cong (\bar h\times\mu(h))^*P$ not necessarily satisfying the cocycle condition.
\end{definition}

Lemma \ref{liftlem} asserts that $F=\FM_P$ lifts as soon as its kernel $P$ is $\mu$-linearisable for some $\mu$.

\begin{definition}\label{mu-equivariant}
We say that an exact functor $F\colon \Db(\cX)\to \Db(\cY)$ is \emph{$\mu$-equivariant} if $F\circ h^*=\mu(h)^*\circ F$ for all $h\in H$. 
\end{definition}

Note that by \cite[Rem.\ 4.1]{CStwisted} a functor $F\colon \Db(\cX)\to\Db(\cY)$ with a left adjoint is a Fourier-Mukai transform with kernel unique up to isomorphism if the following condition holds
\begin{align}\label{FMcond}
\Hom_{\Db(\cY)}(F(A),F(B)[j])=0 \quad\text {for all $A,B\in \Coh(\cX)$ and all $j<0$.} 
\end{align}

\begin{lemma}\label{cyclic-inv-equiv}
 Let $P\in \Db(\cX\times \cY)$ and $F=\FM_P\colon \Db(\cX)\to \Db(\cY)$.
\begin{enumerate}
 \item $P$ is $\mu$-linearisable $\Longrightarrow$ $P$ is $\mu$-invariant. 
\item If $P$ is simple and $H$ is cyclic, the converse of (i) holds.
\item $P$ is $\mu$-invariant $\Longrightarrow$ $F$ is $\mu$-equivariant.
\item If $F$ satisfies condition (\ref{FMcond}), the converse of (iii) holds.  
\end{enumerate}
\end{lemma}
\begin{proof}
The first assertion follows directly by definition and (ii) is Remark \ref{cysimple}. For (iii) we use Lemma \ref{Orlovlem}. It asserts in our situation that 
\begin{align}\label{orlequi}\mu(h)^*\circ F\circ (h^*)^{-1}\cong\FM_{(\bar h\times \mu(h))^*P} \quad\text{for all $h\in H$;}\end{align}
compare Remark \ref{reverse}.
Thus, the existence of an isomorphism $P\cong (\bar h\times \mu(h))^*P$ implies the $\mu$-equivariance of $F$. Finally, if $F$ satisfies condition (\ref{FMcond}), then $\mu(h)^*\circ F\circ (h^*)^{-1}$ does too. Thus, (iv) follows from (\ref{orlequi}) together with \cite[Rem.\ 4.1]{CStwisted}.      
\end{proof}

\begin{remark}
Let $P\in \Db(\cX\times \cY)$ be a kernel of an equivalence, a fully faithful functor, a spherical kernel satisfying (\ref{sphericalcond}), or a $\P^n$-kernel satisfying (\ref{Pcond}).  
Then $P$ is simple by formula (\ref{RFformula}) together with the fact that $\reg_{\Delta\cX}$ is simple; compare \cite[Lem.\ 4]{Ploog-equiv}. Thus, if $H$ is cyclic, the functor $F=\FM_P$ lifts as soon as it is $\mu$-equivariant. Furthermore, note that in this case $P$ has exactly $|\hat H|$ different $\mu$-linearisations. The induced lifts $\widetilde F\colon \Db_H(\cX)\to \Db_{H'}(\cY)$ differ by $\MM_{\rho}$ for $\rho\in \hat H$.  
\end{remark}
\section{Geometric Interpretation}\label{geomint}
By a \textit{Galois cover with group of deck transformations $H$} of smooth projective stacks we mean a quotient morphism $\pi\colon \widetilde \cX\to \cX:=[\widetilde\cX/H]$ for a finite subgroup $H\subset \Aut(\widetilde \cX)$.
For details on quotients of stacks by group actions, we refer to \cite{Rom}.
\begin{definition}\label{geomliftdef}
Let $\pi\colon \widetilde\cX\to\cX$ and $\pi'\colon \widetilde\cY\to\cY$ be Galois covers with groups of deck transformation $H$ and $H'$, respectively, and $\widetilde F\colon \Db(\widetilde \cX)\to \Db(\widetilde \cY)$ and $F\colon
\Db(\cX)\to \Db(\cY)$ be exact functors. 
Then $F$ is called a \textit{descent} of $\tilde F$ (and, equivalently, $\widetilde F$ is called a \textit{lift} of $F$) if the following two conditions hold
\begin{align*}
F\circ \pi_*\cong \pi'_*\circ \widetilde F\colon \Db(\widetilde \cX)\to \Db(\cY), \quad
\pi'^*\circ F\cong \widetilde F\circ \pi^*\colon \Db(\cX)\to\Db(\widetilde\cY). \end{align*}
\end{definition}

By \cite[Thm.\ 4.1]{Rom}, the cover $\pi\colon \widetilde \cX\to \cX=[\widetilde \cX/H]$ is an $H$-torsor. Thus, $\Coh(\cX)\cong \Coh_H(\widetilde\cX)$; see \cite[Thm.\ 4.46]{Vis}.
Under this isomorphism, $\pi^*$ corresponds to $\Res$ and $\pi_*$ corresponds to $\Inf$. Of course, similar considerations apply to $\cY$. Thus, for a functor $F\colon \Db(\widetilde \cX)\to \Db(\widetilde\cY)$ a descent $\Db(\cX)\to \Db(\cY)$ in the sense of Definition \ref{geomliftdef} is the same as an equivariant lift $\Db_H(\tilde\cX)\to \Db_{H'}(\tilde \cY)$ in the sense of Definition \ref{equiliftdef}.   
Furthermore, the objects $\regcan_\cX\otimes\rho$ for $\rho\in \hat H$ correspond to the line bundles on $[\widetilde\cX/H]$ which are in the kernel of $\pi^*\colon \Pic([\widetilde\cX/H])\to \Pic(\widetilde{\cX})$; compare \cite[Sect.\ 7]{Mumbook}. Hence, we can identify $\hat H$ with the subgroup $\ker(\pi^*)\subset\Pic([\widetilde\cX/H])$. Combining the above with Theorem \ref{mainthm} immediately gives 


\begin{prop}\label{descentthm}
Let $\widetilde\cX\to \cX$ and $\widetilde\cY\to \cY$ be Galois covers with groups of deck transformations $H$ and $H'$, respectively. Let $\Phi=\FM_P\colon \Db(\widetilde\cX)\to \Db(\widetilde\cY)$ be a Fourier-Mukai transform. If $P\in \Db(\widetilde \cX\times \widetilde\cY)$ is $\mu$-linearisable for some isomorphism $\mu\colon H\cong H'$, then there is a $\hat \mu$-equivariant descent $ \Psi\colon \Db(\cX)\to \Db(\cY)$. In addition, if $\Phi$ is an equivalence (fully faithful, spherical satisfying (\ref{sphericalcond}), $\P^n$ satisfying (\ref{Pcond})), the same holds for $\Psi$.\qqed 
\end{prop}

If $\pi\colon \widetilde \cX\to \cX$ is a Galois cover with an abelian group of deck transformations $H$, we also have $\Coh(\widetilde\cX)\cong \Coh_{\hat H}(\cX)$ by Proposition \ref{abdual}. Note that in this situation $\pi_*$ is $\Res$ and $\pi^*$ is $\Inf$. The results of the previous subsection again apply and give 

\begin{prop}\label{liftthm}
Let $\pi\colon \widetilde\cX\to \cX$ and $\pi'\colon \widetilde\cY\to \cY$ be Galois covers with abelian groups of deck transformations $H$ and $H'$, respectively. Let $\Phi=\FM_P\colon \Db(\cX)\to \Db(\cY)$ be a Fourier-Mukai transform. If $P\in \Db(\cX\times \cY)$ is $\mu$-linearisable for some isomorphism $\mu\colon \hat H\cong \hat H'$, there is a $\hat \mu$-equivariant lift $\Psi\colon \Db(\widetilde \cX)\to \Db(\widetilde \cY)$. In addition, if $\Phi$ is an equivalence (fully faithful, spherical satisfying (\ref{sphericalcond}), $\P^n$ satisfying (\ref{Pcond})), the same holds for $\Psi$.\qqed
\end{prop}
\begin{remark}\label{simplecyclic}
If $H$ and $H'$ are cyclic and $P$ is simple, which is always the case if the associated FM transform $\Phi$ is fully faithful, one can replace the condition that $P$ is $\mu$-linearisable in Propositions \ref{descentthm} and \ref{liftthm} by the weaker condition that $\Phi$ is $\mu$-equivariant; see Lemma \ref{cyclic-inv-equiv}. 
\end{remark}

Conversely, let $\cX$ be a smooth projective stack and $H\subset\Pic\cX$ a finite subgroup. Let $\cX_H:=\Spec(\cA_H)$ where $\cA_H=\oplus_{\cL\in H}\cL$ is an $\reg_\cX$-algebra with multiplication given by tensor product.
The group of characters $\hat H$ acts on $\cA_H$. The action of $\rho\in \hat H$ is multiplication by $\rho(L)$ on the summands $L$ of $\cA$. The induced action on $\cX_H$ satisfies $[\cX_H/\hat H]\cong \cX$, so that $\cX_H\to\cX$ is a Galois cover with $\hat H$ as the group of deck transformations.       

In particular, if the canonical bundle of $\cX$ is torsion of order $n$, there is the \textit{canonical cover} $\widetilde\cX:=\cX_\omega:=\cX_{\langle \omega_X\rangle}$. It is a cyclic Galois cover of order $n$ of $\cX$ with trivial canonical bundle. 

Note that if $\cY$ is another smooth projective stack together with an equivalence $\Db(\cX)\cong \Db(\cY)$, then the canonical bundle of $\cY$ is also torsion of order $n$. This follows from the fact that equivalences commute with Serre functors; compare \cite[Prop.\ 1.30]{Huy-book}.
 
\begin{cor}\label{cancov}
If $\cX$ and $\cY$ are smooth projective stacks with torsion canonical bundles, then every equivalence $\Phi\colon \Db(\cX)\to \Db(\cY)$ lifts to an equivalence $\Psi\colon \Db(\widetilde \cX)\to \Db(\widetilde \cY)$ between the canonical covers.  
\end{cor}

\begin{proof}
 Since the equivalence $\Phi$ commutes with the Serre functors and shifts, we have $\MM_{\omega_\cY^i}\circ\Phi\cong \Phi\circ\MM_{\omega_\cX^i}$ for every $i\in \Z$. This means that $\Phi$ is $\mu$-equivariant where $\mu\colon \langle \omega_\cX\rangle\cong \langle \omega_\cY\rangle$ is the isomorphism sending $\omega_\cX$ to $\omega_\cY$. The assertion follows by Proposition \ref{liftthm} and Remark \ref{simplecyclic}.   
\end{proof}

\begin{remark}
Note that if $\cX=X$ is a variety, the quotient $[X/H]$ is a variety (namely the quotient variety) if and only if $H$ acts freely on $X$. Thus, for varieties our notion of Galois covers agrees with the usual one. In particular, the criteria of \cite{BM} and \cite{LombPopa} for descent and lift along cyclic covers are obtained as special cases of Propositions \ref{descentthm} and \ref{liftthm}, respectively.
\end{remark}
\section{Applications}\label{Section-Applications}
\subsection{Hochschild homology}
Let $G$ be a finite group acting faithfully on a smooth projective variety $X$ of dimension $d$, let $\cX=[X/G]$ be the associated global quotient stack and write $\omega_{\Delta \cX}$ for $\Delta_*\omega_\cX$. Let 
$\Sigma_\cX:=\omega_{\Delta \cX}[d]\in \Db(\cX\times \cX)$ be the kernel of the Serre functor.

The Hochschild homology $\HH_*(\cX)=\Hom^*(\Sigma_\cX^{-1},\reg_{\Delta \cX})$ decomposes as

\begin{align}\label{HHdecomp}\HH_*(\cX)\cong \bigoplus_{g\in \conj(G)}\HH_*(X^g)^{\Cent(g)}\,;\end{align}
see, for example, \cite[Sect.\ 3]{Popaderived}. Here $\conj(G)$ denotes the set of conjugacy classes in $G$, $\Cent(g):=\Cent_G(g)$ is the centraliser of $g$ in $G$, and $X^g\subset X$ is the fixed point locus. The decomposition can be obtained by the following computation, in which we use the notation $G_\Delta$ for the diagonal action of $G$ on $X\times X$ and the fact that the centraliser of an element $g\in G$ is its stabiliser with respect to the conjugation action of $G$ on itself. Note that $\Hom_{\Db_{\id}(X\times X)}(-,-)\cong \Hom_{\Db_{G_\Delta}(X\times X)}(-,-)$.

\begin{align}\label{decompcomputation}
\begin{aligned}
 \Hom^*_{\Db(\cX\times \cX)}(\omega^{-1}_{\Delta\cX},\reg_{\Delta \cX})&\cong\Hom_{\Db_{G\times G}(X\times X)}(\Inf_{G_\Delta}^{G\times G}\omega^{-1}_{\Delta X},\Inf_{G_\Delta}^{G\times G}\reg_{\Delta X})\\ 
&\cong\Hom_{\Db_{\id}(X\times X)}(\omega^{-1}_{\Delta X},\Res\Inf \reg_{\Delta X})
\\ 
&\cong\Hom_{\Db(X\times X)}(\omega^{-1}_{\Delta X},\oplus_{g\in G}\, (g,1)^*\reg_{\Delta X})^G\\ 
&\cong\bigoplus_{g\in \conj(G)}\Hom_{\Db(X\times X)}(\omega^{-1}_{\Delta X},\reg_{\Gamma_g})^{\Cent(g)}\,.
\end{aligned}
\end{align}
Another computation, which we skip since it is not used in the following, shows that $\Hom_{\Db(X\times X)}(\omega^{-1}_{\Delta X}[-d],\reg_{\Gamma_g})\cong\HH_*(X^g)$. Note that in (\ref{decompcomputation}) we did not start with the Hochschild homology since we omitted the shift. 

Let $\cY=[Y/H]$ be a second smooth projective global quotient stack and $\cP\in \Db(\cX\times \cY)$.
Note that the left and right adjoint kernels are given by $\cP^L=\cP^\vee\star \Sigma_\cY$ and $\cP^R=\Sigma_\cX\star \cP^\vee$.  
The \textit{pull-back} $\cP^*\colon \HH_*(\cY)\to \HH(\cX)$ on Hochschild homology along the kernel $\cP$ is defined by sending $\nu\in\HH_*(Y)=\Hom^*(\Sigma_Y^{-1},\reg_{\Delta Y})$ to the composition 

\begin{align}\label{HHpull}
\begin{aligned}
 &\Sigma_\cX^{-1}\xrightarrow{\Sigma_\cX^{-1}\eta}\Sigma_\cX^{-1}\star\cP^R\star \cP\cong\cP^\vee\star\Sigma_\cY\star\Sigma_\cY^{-1}\star\cP\\&\xrightarrow{\cP^\vee\Sigma_\cY\nu\cP}\cP^\vee\star\Sigma_\cY\star\cP\cong \cP^L\star\cP\xrightarrow\eps\reg_{\Delta\cX}
\end{aligned}
\end{align}
(here the shift by the degree of $\nu$ is omitted in the notation); see \cite[Sect.\ 4.3]{CaWi}.
If $F=\FM_\cP$ satisfies condition (\ref{FMcond}) (so that the FM kernel is unique), it makes also sense to speak of the pull-back $F^*:=\cP^*\colon\HH_*(\cY)\to\HH_*(\cX)$ along $F$ instead. 

Let $U\lhd G$ be a normal subgroup. Then $\conj(U)\subset \conj(G)$.
\begin{lemma}
Under the isomorphism (\ref{HHdecomp}), the only non-zero components of the pull-back $\Inf^*\colon \HH_*([X/G])\to \HH_*([X/U])$ along the inflation functor $\Inf_U^G\colon \Db_U(X)\to \Db_G(X)$ are those of the form  
\[\HH_*(X^u)^{C_G(u)}\to\HH_*(X^u)^{C_U(u)}\quad \text{ for }u\in U.\]
These components are given by the embedding of invariants.
\end{lemma}

\begin{proof}
 The FM kernel of $\Inf=\Inf_U^G$ is given by \[\cP=\Inf_{U_\Delta}^{U\times G}\reg_{\Delta\cX}=\oplus_{g\in G}\reg_{\Gamma_g}\in \Db_{U\times G}(X\times X)\,.\]
The adjoint kernels are the same as $\cP$ just considered as an object of $\Db_{G\times U}(X\times X)$ instead, that is, $\cP^L=\cP^R=\oplus_{g\in G}\reg_{\Gamma_g}\in\Db_{G\times U}(X\times X)$. Hence, in our case the composition (\ref{HHpull}) is given by
\begin{align*}
 \bigoplus_{u\in U} \omega^{-1}_{\Gamma_u}[-d] \to \bigoplus_{g\in G} \omega^{-1}_{\Gamma_g}[-d] \xrightarrow{\cP^L\nu\cP} \bigoplus_{g\in G} \reg_{\Gamma_g}\to \bigoplus_{u\in U} \reg_{\Gamma_u}\,.
\end{align*}
The first morphism is the inclusion of the summands indexed by $U$ and the last morphism is the projection to the summands indexed by $U$. This follows from the fact that these are, up to multiplication by a scalar, the only non-vanishing $U\times U$-equivariant morphisms. Thus, the components of $\Inf^*(\nu)$ coincide with the components of $\cP^L\nu \cP$ on the summands indexed by $U$. Now, the assertion of the lemma can be confirmed by following an element $\nu\in \Hom^*(\omega_{\Delta X}^{-1}[-d],\reg_{\Gamma_g})^{\Cent(g)}\cong\HH_*(X^g)^{\Cent(g)}\subset \HH_*([X/G])$ through the isomorphisms of (\ref{decompcomputation}).
\end{proof}
\begin{cor}\label{Infinj}
The pull-back $\Inf^*\colon \HH_*([X/G])\to \HH_*([X/U])$ is injective if and only if all elements of $G\setminus U$ act freely on $X$.
\end{cor}
\begin{proof}
If all $g\in G\setminus U$ act freely, the corresponding fixed point loci $X^g$ are empty and thus all the components $\HH_*(X^g)^{\Cent(g)}$ in the kernel of $\Inf^*$ vanish. Conversely, if $g\in G\setminus U$ does not act freely, by the Hochschild-Kostant-Rosenberg isomorphism \[\HH_*(X^g)^{\Cent(g)}\cong \Ho^*(X^g, \C)^{\Cent(g)}\cong\Ho^*(X^g/\Cent(g),\C)\neq 0\]
(note that the HKR isomorphism is not graded in the usual sense but this does not cause problems since we only need the non-vanishsing).
\end{proof}
\begin{remark}\label{HHremark}
 The case $U=1$ says that $\Inf\colon \HH_*([X/G])\to \HH_*(X)$ is injective if and only if the quotient stack $[X/G]$ agrees with the quotient
variety $X/G$. Note that in this case $\Inf$ corresponds to the pushforward $\pi_*\colon \Db(X)\to \Db(X/G)$ along the quotient morphism $\pi\colon X\to X/G$ under the equivalence $\Db_G(X)\cong \Db(X/G)$. 
\end{remark}
\subsection{Galois coverings induced by characters}
Let $G$ be a finite group acting on a smooth projective variety $X$.
For a character $\rho\in\hat G$, let $U:=U_\rho:=\ker(\rho\colon G\to \C^*)$. Then $U\lhd G$ is a normal subgroup of index $n:=\ord\rho$ and the canonical morphism $[X/U]\to [X/G]$ is the Galois covering induced by the line bundle $\reg_X\otimes \rho\in \Pic([X/G])$.  
\begin{prop}\label{characterauto}
 Let $Y$ be a smooth projective variety and $\Phi\colon \Db([X/G])\to \Db(Y)$ an equivalence. Let $\rho\in \hat G$ be a character, $U=\ker(\rho\colon G\to \C^*)$, and assume that there is an element of $G\setminus U$ that does not act freely on $X$. Then the induced  autoequivalence $\Phi\circ \MM_\rho\circ \Phi^{-1}$ of $\Db(Y)$ is not given by tensor product with a line bundle on $Y$.  
\end{prop}

\begin{proof}
Assume that $\Phi\circ \MM_\rho\circ \Phi^{-1}\cong \MM_L$ for some line bundle $L\in\Pic(Y)$ and let $\pi\colon Y_L\to Y$ be the associated Galois cover. Clearly, $\ord(L)=\ord(\rho)$ so that there is the isomorphism $\mu\colon \langle \rho\rangle\cong \langle L\rangle$. 
Our assumption says that the equivalence $\Phi$ is $\mu$-equivariant.
Thus, by Proposition \ref{liftthm} it lifts to an equivalence $\Db([X/U])\cong \Db(Y_L)$ which means that there is a commutative diagram

\[\label{HHdiag}\begin{xy}\xymatrix{
\Db([X/U])\ar[r]^\cong\ar[d]_\Inf & \Db(Y_L)\ar[d]^{\pi_*} \\
\Db([X/G]) \ar[r]^\cong & \Db(Y).}
\end{xy} 
\]
Corollary \ref{Infinj} says that $\Inf^*\colon \HH^*([X/G])\to \HH^*([X/U])$ is not injective while $(\pi_*)^*\colon \HH^*(Y)\to\HH^*(Y_L)$ is injective by Remark \ref{HHremark}. Considering the commutative diagram (\ref{HHdiag}), this contradicts the functoriality of the pull-back in Hochschild homology; see \cite[Thm.\ 6]{CaWi}.  
\end{proof}
\begin{remark}
When the action of $G$ on $X$ satisfies the assumptions of the Bridgeland-King-Reid theorem \cite{BKR}, the above yields non-standard autoequivalences of the crepant resolution $Y=\Hilb^G(X)$ of the quotient variety $X/G$. Indeed, let $\rho\in \hat G$ be a character and let $F_\rho:= \Phi\circ \MM_\rho\circ \Phi^{-1}$ where $\Phi\colon\Db([X/G])\cong \Db(Y)$ is the BKR equivalence. Assume that $F_\rho$ is standard, that is $F_\rho=\MM_L\circ\phi_*[m]$ for some $L\in \Pic(Y)$, $\phi\in\Aut(Y)$, and $m\in\Z$. There is an open subset $U\subset Y$ such that skyscraper sheaves of points on $U$ correspond under $\Phi$ to skyscraper sheaves of free orbits. Since skyscraper sheaves of free orbits lie in the image of the inflation functor, they are invariant under $\MM_\rho$. It follows that $F_\rho(\C(u))=\C(u)$ for $u\in U$ which shows that $\phi=\id$ and $m=0$. Thus, we are left with $F_\rho=\MM_L$ which is ruled out (in the case that $G\setminus U_\rho$ does not act trivially) by 
the above proposition.

The above also works in the more general setting of $G$-constellations, see, for instance, \cite{Craw-Ishii} for this notion.
\end{remark}
\subsection{Stacks with characters as canonical bundles}
\begin{prop}\label{noFMpartners}
Let $X$ be a smooth projective variety with trivial canonical bundle and $G\subset \Aut(X)$ a finite subgroup such that the canonical bundle is not trivial as a $G$-bundle, that is, $\omega_{[X/G]}\cong\reg_X\otimes \rho$ for some non-trivial character $\rho$. Then there is no smooth projective variety $Y$ such that $\Db([X/G])\cong \Db(Y)$. 
\end{prop}
\begin{proof}
Assume that there exists a smooth projective variety $Y$ and an equivalence $\Phi\colon \Db([X/G])\cong \Db(Y)$. 
Since any equivalence commutes with shifts and Serre functors, we have
$\Phi\circ \MM_\rho\circ \Phi^{-1}\cong \MM_{\omega_Y}$ which contradicts Proposition \ref{characterauto}.
\end{proof}

\begin{remark}
 In some very weak sense, this can be seen as a converse of the Bridgeland-King-Reid theorem \cite{BKR} which gives sufficient conditions for $[X/G]$ to be derived equivalent to a special smooth variety, namely the Nakamura G-Hilbert scheme. One of the two conditions is that $\omega_X$ is locally trivial as a $G$-bundle. The above says that, under the additional assumption that $\omega_X$ is trivial as a non-equivariant bundle, this condition is also neccesary for $[X/G]$ to be derived equivalent to any smooth variety. 
\end{remark}

\subsection{Symmetric quotients and generalised Kummer stacks}
We introduce our two main example series. Let $X$ be any smooth projective variety and $n\ge 2$. The symmetric group $\sym_n$ acts on $X^n$ by permutation of the factors. We call the corresponding global quotient stack $\cS^nX:=[X^n/\sym_n]$ the \textit{symmetric quotient stack}. The most relevant case is the one when $X$ is a surface, where $\cS^nX$ is derived equivalent to the Hilbert scheme $X^{[n]}$ of $n$ points on $X$ by \cite{BKR} and \cite{Haiman}; see Subsection \ref{BKRsect} for details.
 
In the case that $X=A$ is an abelian variety, $A^n$ contains the $\sym_n$-invariant subvariety
\[
N_{n-1}A:=\{(a_1,\ldots,a_n)\mid a_1+\ldots +a_n=0\}\cong A^{n-1}. 
\]
We call $\cK_{n-1}A:=[N_{n-1}A/\sym_n]$ the \textit{generalised Kummer stack}. The reason for the name is that in the case that $A$ is an abelian surface, $\cK_{n-1}A$ is derived equivalent to the generalised Kummer variety $K_{n-1}A\subset A^{[n]}$; see \cite[Lem.\ 6.2]{Meachan}. 

If $\omega_X$ is trivial, the canonical bundle of the product $\omega_{X^n}\cong \omega_X^{\boxtimes n}$ is trivial too. Similarly, for any abelian variety $A$, the canonical bundle of $N_{n-1}A$ is trivial. 

There is the following fundamental difference depending on the parity of the dimensions of $X$ and $A$. Namely, in the even dimensional case also the equivariant canonical bundles $\omega_{\cS^nX}\in \Coh(\cS^nX)\cong \Coh_{\sym_n}(X^n)$ and $\omega_{\cK_{n-1}A}\in \Coh(\cK_{n-1}A)\cong \Coh_{\sym_n}(N_{n-1}A)$ are trivial. 

But in the odd dimensional case we have

\begin{lemma}\label{oddsign}
If $\dim X$ and $\dim A$ are odd, we have $\omega_{\cS^nX}\cong \reg_{X^n}\otimes \alt$ and $\omega_{[N_{n-1}A/\sym_n]}\cong \reg_{N_{n-1}A}\otimes \alt$ where $\alt$ denotes the sign representation of $\sym_n$, that is, the unique non-trivial character. 
\end{lemma}

\begin{proof}
By Lemma \ref{simple} there are only two possible linearisations of $\omega_{X^n}$ so that the linearisation of the canonical bundle is determined by the $\sym_n$-action on the fibre $\omega_{X^n}(p)$ over a point $p=(x,\ldots,x)$ on the small diagonal. We denote by $\C^{\{1,\ldots,n\}}$ the permutation representation of $\sym_n$ and by $\rho_n$ the standard representation, that is the quotient of $\C^{\{1,\ldots,n\}}$ by the one dimensional invariant subrepresentation, so that $\C^{\{1,\ldots,n\}}\cong \rho_n\oplus \C$. 
We have $\det \C^{\{1,\ldots,n\}}\cong \det \rho_n\cong \alt$; see \cite[Prop.\ 2.12]{FulHar}.
Furthermore, $\Omega_{X^n}(p)\cong(\C^{\{1,\ldots,n\}})^{\oplus d}$ as $\sym_n$-representations where $d=\dim X$. Thus, $\omega_{X^n}(p)\cong \det \Omega_{X^n}(p)\cong \alt^{\otimes d}$. Similarly, $\Omega_{N_{n-1}A}(p)\cong \rho_n^{\oplus d}$ and thus also $\omega_{N_{n-1}A}(p)\cong \alt^{\otimes d}$ for $d=\dim A$.
\end{proof}  
\begin{prop}
\begin{enumerate}
\item
Let $X$ be a smooth projective variety of odd dimension with trivial canonical bundle and $n\ge 2$. Then the symmetric quotient stack $\cS^nX=[X/\sym_n]$ is not derived equivalent to any smooth projective variety.
\item
Let $A$ be an abelian variety of odd dimension and $n\ge 2$. Then the generalised Kummer stack $\cK_{n-1}A=[N_{n-1}A/\sym_n]$ is not derived equivalent to any smooth projective variety.
\end{enumerate}
\end{prop}
 
\begin{proof}
This follows from Proposition \ref{noFMpartners} together with Lemma \ref{oddsign}.
\end{proof} 
\subsection{Kummer stacks}
Let us consider the generalised Kummer stack in the case $n=2$. Then $N_1A\subset A^2$ is just the anti-diagonal and thus isomorphic to $A$. Under the isomorphism $N_1A\cong A$ the non-trivial element of $\sym_2$ acts by $\iota\colon A\to A$, $x\mapsto -x$.
The global quotient stack $[A/\iota]\cong \cK_1A$ is called the \textit{Kummer stack} associated to $A$. Note that if $A$ is an abelian surface, there is the derived McKay correspondence $\Db_\iota([A])\cong \Db([A/\iota])\cong \Db(K(A))$ where $K(A)$ is the Kummer K3 surface associated to $A$. It is proven in \cite{Ste} that for every derived equivalence $\Db(A)\cong \Db(B)$ between abelian varieties there is a derived equivalence $\Db([A/\iota])\cong \Db([B/\iota])$. In the case that $\dim A=\dim B$ is odd, also the converse holds:

\begin{prop}
 Let $A$ and $B$ be abelian varieties of odd dimension. Then $\Db([A/\iota])\cong \Db([B/\iota])$ implies $\Db(A)\cong \Db(B)$. 
\end{prop}

\begin{proof}
As discussed in the previous subsection, the canonical bundle of $[A/\iota]$ is torsion of order 2 and the canonical cover is $\pi\colon A\to [A/\iota]$. Hence, this is just a special case of Corollary \ref{cancov}.
\end{proof}
\subsection{Enriques quotients of Hilbert schemes}\label{BKRsect}
Let $X$ be a smooth projective surface and $n\in \N$.
Recall that there is the Bridgeland-King-Reid, Haiman (in the following abbreviated as BKR-H) equivalence 
\[
 \Phi=\FM_{\reg_{I^nX}}\colon \Db(X^{[n]})\xrightarrow\cong  \Db_{\sym_n}(X^n)\cong \Db(\cS^nX)
\]
where $I^nX=(X^{[n]}\times_{S^nX}X^n)_{\mathsf{red}}\subset X^{[n]}\times X^n$ is the \textit{isospectral Hilbert scheme}; see \cite{BKR} and \cite{Haiman}. The fibre product is given by the quotient morphism $\pi\colon X^n\to S^nX:=X^n/\sym_n$ and the Hilbert--Chow morphism $\HC\colon X^{[n]}\to S^nX$. 

Any automorphism $\phi\in \Aut(X)$ induces automorphisms $\phi^{[n]}\in \Aut(X^{[n]})$, $\phi^{\times n}\in \Aut(X^n)$, and $S^n\phi\in \Aut(S^nX)$. The morphisms $\HC$ and $\pi$ commute with these induced automorphisms. It follows that $I^nX$ is invariant under $\phi^{[n]}\times\phi^{\times n}$. 

Let $\phi$ be of finite order. Then $\phi^{[n]}$ and $\phi^{\times n}$ are of the same order and there is the isomorphism $\mu\colon \langle\phi^{[n]}\rangle\cong \langle\phi^{\times n}\rangle$ given by sending one generator to the other. Since $I^nX$ is invariant under $\phi^{[n]}\times \phi^{\times n}$, the object $\reg_{I^nX}$ carries a canonical $\mu$-linearisation. Thus, Proposition \ref{descentthm} gives

\begin{cor}
The BKR-H equivalence $\Phi\colon  \Db(X^{[n]})\xrightarrow\cong  \Db(\cS^nX)$ descends to
an equivalence
\[
\check \Phi\colon \Db([X^{[n]}/\phi^{[n]}])\xrightarrow\cong \Db([\cS^nX/\phi^{\times n}]) 
\]
between the derived categories of the quotient stacks.\qqed
\end{cor}

Another important functor occurring in the setup of Hilbert schemes of points on surfaces is the Fourier-Mukai transform $F_n=\FM_{\cI_\Xi}\colon \Db(X)\to \Db(X^{[n]})$ along the universal ideal sheaf $\cI_{\Xi}$. Clearly, the universal family $\Xi\subset X\times X^{[n]}$ is invariant under $\phi\times \phi^{[n]}$. Hence, $\cI_\Xi\in \Db(X\times X^{[n]})$ carries a canonical $\phi\times \phi^{[n]}$-linearisation. Therefore, we get

\begin{cor}
For any surface $X$ the Fourier-Mukai transform $F=\FM_{\cI_\Xi}$ with the structure sheaf of the universal family as kernel descends to 
a functor $\check F_n\colon \Db([X/\phi])\to\Db([X^{[n]}/\phi^{[n]}])$.\qqed
\end{cor}

There are derived category versions of the Nakajima operators as defined in \cite{Krug-naka} given by Fourier-Mukai transforms 
\[H_{\ell,n}\colon \Db(X\times X^{[n]})\cong \Db_{\sym_\ell}(X\times X^\ell)\to \Db_{\sym_{n+\ell}}(X^{n+\ell})\cong \Db(X^{[n+\ell]})\,.\]
Basically, they are constructed using pushforwards along the closed embeddings $\delta_I\colon X\times X^{\ell}\to X^{n+\ell}$ onto the partial diagonals $\Delta_I=\{(x_1,\ldots,x_n)\mid x_i=x_j\forall i,j\in I\}$ for $I\subset \{1,\ldots,n+\ell\}$ with $|I|=n$. The graph of $\delta_I$ is invariant under the morphism $\phi\times\phi^{\times\ell}\times \phi^{\times n+\ell}$. Hence, we get

\begin{cor}
The Nakajima operators $H_{\ell,n}$ descend to
\[\check H_{\ell,n}\colon \Db\big([\frac{X\times X^{[n]}}{\phi\times \phi^{[\ell]}}]\big)\cong \Db_{\sym_\ell}\big([\frac{X\times X^\ell}{\phi \times\phi^{\times \ell}}]\big)\to \Db_{\sym_{n+\ell}}\big([\frac{X^{n+\ell}}{\phi^{\times n+\ell}}]\big)\cong \Db\big([\frac{X^{[n+\ell]}}{\phi^{[n+\ell]}}]\big).\]
For $n\ge \max\{\ell,2\}$ the $H_{\ell,n}$ are $\P^{n-1}$-functors satisfying (\ref{Pcond}); see \cite{Krug-naka}.
Therefore, the descents $\check H_{\ell,n}$ are $\P^{n-1}$-functors too.\qqed
\end{cor}

Let $\widetilde X$ be a K3 surface with a fixed point free involution $\tau\colon X\to X$ and
let $X=\widetilde X/\tau$ be the corresponding Enriques surface.
For $n\in \N$ an odd number, the induced automorphism $S^n\tau\colon S^n\widetilde X\to S^n\widetilde X$ is a fixed point free involution. Hence, the same holds for $\tau^{[n]}$ and $\tau^{\times n}$. The quotient $X_{[n]}:=\widetilde X^{[n]}/\tau^{[n]}$ is the first example of a higher dimensional Enriques manifold in the sense of \cite{Oguiso-Schroer}. Note that $\tau^{\times n}$ commutes with the $\sym_n$-action on $\widetilde X^n$. Thus there is an induced $\sym_n$-action on $X_n:=\widetilde X^n/\tau^{\times n}$. By the above discussion, we get an equivalence 
\[\check\Phi\colon \Db(X_{[n]})\xrightarrow \cong \Db_{\sym_n}(X_n)\,.\]
In fact, this itself is a special case of the BKR theorem. Indeed, one can show quite easily that $X_{[n]}\cong \Hilb^{\sym_n}(X_n)$.

In \cite{Addington} it is proved that for a K3 surface $\widetilde X$ the Fourier-Mukai transform $F_n\colon\Db(\widetilde X)\cong \Db(\widetilde X^{[n]})$ is a $\P^{n-1}$-functor. Therefore, the descent $\check F_n\colon \Db(X)\to \Db(X_{[n]})$ is again a $\P^{n-1}$-functor by Proposition \ref{descentthm}.

Further examples of Enriques manifolds are given by quotients of generalised Kummer varieties; see \cite{Oguiso-Schroer}, \cite{BNWS}. Again, the quotients come from fixed point free automorphisms of finite order on the Kummer varieties which are naturally induced by automorphisms of the abelian surface. Therefore, the BKR-H equivalence descends to the corresponding quotients.

Furthermore, the Fourier-Mukai transform along the universal ideal sheaf is a $\P^n$-functor for generalised Kummer varieties; see \cite{Meachan}. It descends to a $\P^n$-functor from the derived category of the hyperelliptic quotient of the abelian surface to the derived category of the Enriques quotient of the generalised Kummer variety.

\subsection{Calabi--Yau covers of Hilbert schemes}\label{CYcover}

Let $X$ be a smooth projective surface.
The line bundle $L^{\boxtimes n}\in \Pic(X^n)$ carries a canonical $\sym_n$-linearisation.
There is also the induced line bundle $L_n:=\HC^*((\pi_*L^{\boxtimes n})^{\sym_n})$ (here we use notation from the previous subsection). 
Writing again $\Phi$ for the BKR-H equivalence, we have
\begin{align}\label{lineBKR}
\Phi\circ \MM_{L_n}\cong \MM_{L^{\boxtimes n}}\circ\Phi\colon \Db(X^{[n]})\to \Db_{\sym_n}(X^n)\,; 
\end{align}
see, for example, \cite[Lem.\ 9.2]{Krug-tensor}.

Let now $L$ be of finite order. By (\ref{lineBKR}) together with Proposition \ref{liftthm} and Remark \ref{simplecyclic} it follows that $\Phi$ lifts to an equivalence
\[
 \widetilde \Phi\colon \Db((X^{[n]})_{L_n})\xrightarrow \cong \Db_{\sym_n}((X^n)_{L^{\boxtimes n}})\cong \Db((\cS^nX)_{L^{\boxtimes n}})\,.
\]
Again, this equivalence is a special case of the BKR theorem as $(X^{[n]})_{L_n}\cong \Hilb^{\sym_n}((X^n)_{L^{\boxtimes n}})$. 

Also, there is the relation $\MM_{L^{\boxtimes (n+\ell)}}\circ H_{\ell,n}\cong H_{\ell,n}\circ \MM_{L^n\boxtimes L^{\boxtimes \ell}}$. This shows that for $\ell\ge 1$ there is a lift to a $\P^{n-1}$-functor 
\[
 \widetilde H_{\ell,n}\colon \Db((X\times X^{[\ell]})_{L^n\boxtimes L_\ell})\to \Db(X^{[n+\ell]}_{L_{n+\ell}})\,. 
\]

\begin{remark}
Probably the most interesting special case is when $X$ is an Enriques surface and $L=\omega_X$ is the canonical bundle. We have $(\omega_X)_n=\omega_{X^{[n]}}$ and the canonical cover $\CY_n(X):=\widetilde{X^{[n]}}:=X^{[n]}_{\omega_{X^{[n]}}}$ is a Calabi--Yau variety; see \cite[Prop.\ 1.6]{NieperW-twisted}.

There is a second $2n$-dimensional Calabi-Yau variety induced by an Enriques surface $X$. Namely, the canonical cover $\widetilde \CY_n(X):=\widetilde{X^n}$ of the Cartesian product. 
Indeed, a smooth projective variety $M$ is Calabi-Yau if and only if $\reg_M$ is a spherical object. Furthermore, $\reg_{X^n}=\reg_X^{\boxtimes n}$ is exceptional and hence $\reg_{\widetilde{X^n}}=\pi^*\reg_{X^n}$ is spherical; see \cite[Prop.\ 3.13]{Seidel-Thomas} or \cite[Rem.\ 3.11]{KSos}. We have $\Db(\CY_n(X))\cong \Db(\widetilde\CY_n(X))^{\sym_n}$. 

More generally, let $X_1,\ldots ,X_m$ be smooth projective varieties with torsion canonical bundles of order 2 such that the structure sheaves $\reg_{X_i}$ are exceptional. Then $\reg_{X_1\times \ldots\times X_m}$ is again exceptional and thus the canonical cover $\widetilde \CY(X_1,\ldots,X_m):=\widetilde{X_1\times \ldots\times X_m}$ is Calabi-Yau.

Going the other way around, there are Calabi-Yau varieties with fixed point free involutions, so the quotient has canonical bundle of order $2$ and its structure sheaf is an exceptional object, see \cite[Subsect.\ 4.3]{BNWS}.
\end{remark}

\appendix
\section{Necessary condition for lifts}
In this section we will prove a converse of Theorem \ref{liftthm}, namely
\begin{prop}\label{neccond}
Let $\pi_1\colon\widetilde \cX\to \cX$ and $\pi_2\colon \widetilde \cY\to \cY$ be Galois covers of smooth projective stacks with abelian groups of deck transformations $G$ and $G'$, respectively, and let $\Phi=\FM_P\colon \Db(\cX)\to \Db(\cY)$ be a fully faithful functor. Then, for $\Phi$ to lift to a fully faithful functor $\Psi\colon\Db(\tilde \cX)\to \Db(\tilde \cY)$ it is necessary that the FM kernel $P\in \Db(\cX\times \cY)$ is $\mu$-linearisable for some isomorphism $\mu\colon \hat G\cong \hat G'$.   
\end{prop}
Note that for $\cL\in \hat G\subset\Pic(\cX)$ the FM kernel of the autoequivalence $\MM_\cL$ is given by $\Delta_*\cL\in \Db(\cX\times \cX)$. 
\begin{lemma}\label{Delsimple}
For $\cL_1\neq \cL_2\in \hat G$ we have $\Hom_{\Db(\cX\times \cX)}(\Delta_*\cL_1,\Delta_*\cL_2)=0$.
\end{lemma}
\begin{proof}
We have $\Hom(\Delta_*\cL_1,\Delta_*\cL_2)\cong \Hom(\cL_1,\cL_2)\cong \Gamma(\cX,\cL_1^{-1}\otimes \cL_2)$. Since $\cL_1^{-1}\otimes \cL_2$ is a non-trivial torsion line bundle, it does not have non-zero global sections. \end{proof}

\begin{lemma}\label{nohoms}
Let $P\in \Db(\cX\times\cY)$ such that $\Phi=\FM_P\colon \Db(\cX)\to\Db(\cY)$ is fully faithful. 
Then $\Hom_{\Db(\cX\times \cY)}((\cL_1\boxtimes \reg)\otimes P,(\cL_2\boxtimes \reg)\otimes P)=0$ for $\cL_1\neq\cL_2\in\hat G$. 
\end{lemma}
\begin{proof}
Since $\Phi\circ \MM_{\cL_1}$ is fully faithful, its FM kernel is unique up to isomorphism. Thus, by Lemma \ref{Orlovlem} we have $P\star \Delta_*\cL_1\cong (\cL_1\boxtimes \reg)\otimes P$. Since $\Phi$ is fully faithful, $P^R\star P\cong \reg_{\Delta\cX}$. By (\ref{adjointformula}) together with the previous lemma 
\begin{align*}
 \Hom((\cL_1\boxtimes \reg)\otimes P,(\cL_2\boxtimes \reg)\otimes P)&\cong \Hom(P\star \Delta_*\cL_1, P\star \Delta_*\cL_1)\\&\cong \Hom(\Delta_*\cL_1,\Delta_*\cL_2)=0\,.\qedhere
\end{align*}
\end{proof}    
\begin{proof}[Proof of Proposition \ref{neccond}]
Let $\Psi\colon \Db(\widetilde\cX)\to \Db(\widetilde \cY)$ be a lift of $\Phi$ with FM kernel $\cQ\in \Db(\widetilde\cX\times \widetilde\cY)$. By the definition of  a lift, we have $\Phi \pi_{1*} \pi_1^*\cong\pi_{2*}\Psi\pi_1^*\cong \pi_{2*}\pi_2^*\Phi$. The functor $\Phi \pi_{1*} \pi_1^*\cong \oplus_{\cL\in \hat G} \Phi\circ\MM_\cL$ satisfies (\ref{FMcond}) so that its FM kernel is unique up to isomorphism. By Lemma \ref{Orlovlem} the FM kernels of $\pi_{2*}\Psi\pi_1^*$ and $\pi_{2*}\pi_2^*\Phi$ are $Q:=(\pi\times \pi)_*\cQ$ and $\oplus_{\cL'\in \hat G'}(\reg\boxtimes \cL')\otimes P$, respectively. Hence, \begin{align}\label{iso}Q\cong\oplus_{\cL'\in \hat G'}(\reg\boxtimes \cL')\otimes P\,.\end{align}
Because of the identification $\Db(\tilde \cX\times \tilde \cY)\cong \Db_{\hat G\times \hat G'}(\cX\times \cY)$, the object $Q$ carries a $\hat G\times \hat G'$-linearisation $\nu$. By Lemma \ref{nohoms}, for every $\cL\in\hat G$ there is a unique $\mu(\cL)\in \hat G'$ such that the isomorphism $\nu_{(\cL^{-1},\reg)}\colon Q\to (\cL^{-1}\boxtimes \reg)\otimes Q$ restricts under (\ref{iso}) to an isomorphism 
\begin{align}\label{mulin}P\to (\cL^{-1}\boxtimes \reg)\otimes (\reg\boxtimes \mu(\cL))\otimes P\cong (\cL^{-1}\boxtimes \mu(\cL))\otimes P\,.\end{align} 
By the uniqueness of $\mu(\cL)$, the map $\mu$ is a homomorphism and (\ref{mulin}) defines a $\mu$-linearisation of $P$.       
\end{proof}
Analogously, one can prove a partial converse of Theorem \ref{descentthm}.
%


\end{document}